\documentclass[11pt]{amsart}
\usepackage{amsmath}
\usepackage{amssymb}
\usepackage{amscd}

\def\NZQ{\mathbf}               % the font for N,Z,Q,R,C
\def\NN{{\NZQ N}}
\def\QQ{{\NZQ Q}}
\def\ZZ{{\NZQ Z}}
\def\RR{{\NZQ R}}

\def\PP{{\NZQ P}}

\def\P{\mathcal P}

\newtheorem{theorem}{Theorem}[section]
\newtheorem{proposition}[theorem]{Proposition}

\theoremstyle{definition}
\newtheorem{definition}[theorem]{Definition}

\newtheorem{Theorem}{Theorem}[section]
\newtheorem{Lemma}[Theorem]{Lemma}
\newtheorem{Corollary}[Theorem]{Corollary}
\newtheorem{Proposition}[Theorem]{Proposition}
\newtheorem{Remark}[Theorem]{Remark}

\let\epsilon\varepsilon
\let\phi=\varphi
\let\kappa=\varkappa
\begin{document}

\title{Semigroups of valuations on local rings}

\author{Steven Dale Cutkosky  and Bernard Teissier}
\thanks{The first author was partially supported by NSF and the second by the Miller Scholar in residence program of the University of Missouri at Columbia}

%\begin{abstract}
%\end{abstract}

\maketitle
\section*{Introduction}  In recent years the progress and applications of valuation theory have brought to light the importance of understanding the semigroups of values which valuations take on a noetherian ring $(R,m)$ contained in the ring $(R_\nu,m_\nu)$ of the valuation. Only two general facts are known: \par\noindent
The first is that these semigroups are well ordered subsets of the positive part of the value group, of ordinal type at most $\omega^h$, where $\omega$ is the ordinal type of the well ordered set $\NN$ and $h$ is the rank of the valuation (see \cite{ZS}, Appendix 3, Proposition 2). Being well ordered, each value semigroup of a noetherian ring has a unique minimal system of generators, indexed by an ordinal no greater than $\omega^h$. \par\noindent
The second is, in the case where $m_\nu\bigcap R=m$, the Abhyankar inequality
$$\hbox{\rm rr}(\nu)+\hbox{\rm tr}_kk_\nu\leq \hbox{\rm dim}R,$$
between the rational rank of the group of the valuation, the transcendence degree over the residue field of $R$ of the residue field of $R_\nu$ and the dimension of $R$ (see \cite{ZS}, Vol. 2, Appendix 2).\par\noindent

Examples, starting with plane branches (see \cite{Z}  and \cite{Te}), suggest that the structure of the semigroup contains important information on the process of local uniformization. \par
In this paper we shall consider mostly valued noetherian rings $(R,m)\subset (R_\nu, m_\nu)$ such that $m_\nu \bigcap R=m$ and the residual extension $R/m\to R_\nu/m_\nu$ is trivial. We call such valuation rational valuations of $R$. In this case, the finite generation of the semigroup is equivalent to the finite generation over $k_\nu=R_\nu/m_\nu$ of the graded algebra associated to the valuation $$\hbox{\rm gr}_{\nu}R=\bigoplus_{\gamma\in \Gamma} \P_{\gamma}/\P^+_{\gamma},$$
where $$\P_{\gamma}=\{f\in R\vert \nu(f)\geq \gamma\},\ \ \ \P^+_{\gamma}=\{f\in R\vert \nu(f)> \gamma\}$$ and $\Gamma$ is the value group of the valuation.\par\noindent
This equivalence is true because each homogeneous component $\P_{\gamma}/\P^+_{\gamma}$ is non zero exactly if $\gamma$ is in the value semigroup $S=\nu (R\setminus\{0\})$, and then it is a one-dimensional vector space over $R/m=R_\nu/m_\nu$ (see \cite{Te}, section 4).\par
If the semigroup is finitely generated, it is not difficult to find a valuation of a local noetherian domain which has it as semigroup: the semigroup algebra with coefficients in a field $K$ is then a finitely generated algebra. We can define a weight on its monomials by giving to each generator as weight the generator of the semigroup to which it corresponds. Then we define a valuation by deciding that the valuation of a polynomial is the smallest weight of its monomials.\par
The simplest example is that of a subsemigroup of $\NN$; it is necessarily finitely generated (see \cite{Re}, Th. 83 p. 203) and by dividing by the greatest common divisor we may assume that its generators $\gamma_1,\ldots ,\gamma_g$ are coprime, so that it generates the group $\ZZ$. It is the semigroup of values of the $t$-adic valuation on the ring $K[t^{\gamma_1},\ldots ,t^{\gamma_g}]\subset K[t]$ of the monomial curve corresponding to the semigroup.\par
In section \ref{z2} we give an example of a semigroup of values of a valuation on a polynomial ring $L[x,y,z]$ over a field which generates the group $\ZZ^2$ but is not finitely generated as a semigroup. Moreover, the smallest closed cone in $\RR^2$ with vertex $0$ containing the semigroup is rational.\par
  If the semigroup is not finitely generated, very little is known about its structure. While we know that all totally ordered abelian groups of finite rational rank appear as value groups of valuations of rational function fields centered in a polynomial ring, we do not know any general condition implying that a well ordered sub-semigroup of the positive part of a totally ordered group of finite rational rank appears as semigroup of values of a noetherian ring. We shall see in section \ref{ratrk1} a characterization of those well-ordered sub-semigroups of $\QQ_+$ which are the semigroups of values of a $K$-valuation on $K[X,Y]_{(X,Y)}$, but no general result is known for subsemigroups contained in the positive part $\QQ^2_{\succcurlyeq 0}$  of $\QQ^2$ for some total ordering $\succcurlyeq$.\par
We show in section \ref{accu} that a sub-semigroup $S$ of the positive part of a
totally ordered abelian group of finite rational rank, such that $S$ has  ordinal type
$<\omega^{\omega}$,  has no accumulation points, and $S$ has ordinal type $\le \omega^h$, where $h$ is the rank of the group generated by $S$.
These are properties which are held by the semigroup of a valuation on a noetherian ring.\par
Even if one could solve the problem of determining which semigroups of a totally ordered group of (real) rank one come from noetherian rings, an induction on the rank would meet serious difficulties, and one of the questions which appear naturally is that of the position of the generators of the semigroup in regard to the valuation ideals of the valuations with which the given valuation is composed. \par
More precisely, let $\nu$ be a valuation taking non negative values on the noetherian local domain $R$, let $\Psi$ is the convex subgroup of real rank one in the group $\Gamma$ of the valuation, and let $P$ be the center in $R$ of the valuation $\nu_1$ with values in $\Gamma/\Psi$ with which $\nu$ is composed. The valuation $\nu$ induces a residual valuation $\overline \nu$ on the quotient $R/P$ and when the image $q(\gamma)$ of $\gamma$ in $\Gamma/\Psi$ by the canonical quotient map $q\colon \Gamma\to \Gamma/\Psi$ is $\gamma_1$ we have inclusions of the valuation ideals corresponding to $\nu$ and $\nu_1$:
$$\P_{\gamma_1}^+\subset \P_{\gamma}\subset \P_{\gamma_1},$$
\par
Since $R$ is noetherian, the quotients $\P_{\gamma_1}/\P_{\gamma_1}^+$ are finitely generated $R/P$-modules for all $\gamma_1\in \Gamma/\Psi$. Each is endowed by the filtration  $\mathcal F (\gamma_1)$ by the $({\mathcal F}_{\gamma}=\P_{\gamma}/\P_{\gamma_1}^+)_{\gamma\in q^{-1}(\gamma_1)}$ with a structure of filtered $R/P$-module with respect to the filtration $\mathcal F (0)$ of $R/P$ by its valuation ideals $(\overline{\P_{\delta}})_{\delta\in \Psi}$, where $\overline{\P_{\delta}}=\P_\delta R/P=\P_\delta/P$.\par
One could hope (see \cite{Te}) that the associated graded module $$\hbox{\rm gr}_{\mathcal F(\gamma_1)}\P_{\gamma_1}/\P_{\gamma_1}^+=\bigoplus_{\gamma\in q^{-1}(\gamma_1)}\P_{\gamma}/\P_{\gamma}^+$$ is finitely generated over the associated graded ring $$\hbox{\rm gr}_{\overline \nu}R/P=\bigoplus_{\delta\in \Psi_+}\P_{\delta}/\P^+_{\delta}.$$ For rational valuations, this would be equivalent to the fact that for each $\gamma_1>0$ only finitely many new generators of the semigroup appear in $\P_{\gamma_1}\setminus \P^+_{\gamma_1}$. It would set a drastic restriction on the ordinal type of the minimal set of generators of the semigroup $\nu (R\setminus \{0\})$.\par
In sections \ref{z7} and \ref{z8}, we give examples showing that this is not at all the case, and that in fact the cardinality of the set of new generators may vary very much with the value of $\gamma_1$.
One might have hoped that this lack of finiteness is due to some transcendence in the residual extension from $R/P$ to $R_\nu/m_{\nu_1}$ and disappears after some birational extension of $R$ to another noetherian local ring contained in $R_\nu$ which absorbs the transcendence. In one of the examples (section \ref{z8}) there is no residual extension.\par
The conclusion is that the semigroups of values of noetherian rings do not seem to be subject to more constraints than the two stated at the beginning.
\section{A criterion for finite generation}
Given a commutative semigroup $S$, a set $M$ is endowed with a structure of an $S$-module by an operation $S\times M\to M$ written $(s,m)\mapsto s+m$ such that $(s+s')+m=s+(s'+m)$. Recall that $M$ is then said to be a finitely generated $S$-module if there exist finitely many elements $m_1,\ldots, m_k$ in $M$ such that $M=\bigcup_{i=1}^k (S+m_i)$.
\begin{Proposition}\label{Theorem5}  Keeping the notations of the previous section, suppose that $K$ is a field and $R$ is a local domain with residue field $K$ dominated by a valuation ring $R_\nu$ of the field of fractions of $R$. Assume that the residual extension $R/m\to R_\nu/m_\nu$ is trivial.
Then for $\gamma_1\in \Gamma/\Psi$, The associated graded module
$$
\mbox{\rm gr}_{{\mathcal F}(\gamma_1)}{\mathcal P}_{\gamma_1}/{\mathcal P}^+_{\gamma_1}
$$
is   a  finitely generated  $\mbox{\rm gr}_{\overline\nu}R/P$-module
if and only if
$$
F_{\gamma_1}=\{\nu(f)\mid f\in R\mbox{ and }\nu_1(f)=\gamma_1\}
$$
 is a finitely generated  module over the semigroup $\overline F=\overline \nu(R/P\setminus \{0\})$.
\end{Proposition}
\begin{proof} Let us first remark that since ${\mathcal P}_{\gamma_1}/{\mathcal P}^+_{\gamma_1}$ is an $R/P$-module, the set $F_{\gamma_1}$ is an $\overline F$-module. \par
We know from (\cite{Te}, 3.3-3.5, 4.1) that the non zero homogeneous components of the graded algebra $\mbox{gr}_\nu R =\bigoplus_{\gamma\in \Gamma}{\mathcal P}_{\gamma}/{\mathcal P}^+_{\gamma}$ are one-dimensional $K$-vector spaces whose degrees are in bijection with $F=\nu (R\setminus \{0\})$, and that $\overline F=F\bigcap \Psi$. \par
 The $\mbox{gr}_{\overline\nu}R/P$-module $
\mbox{gr}_{{\mathcal F}(\gamma_1)}{\mathcal P}_{\gamma_1}/{\mathcal P}^+_{\gamma_1}
$ is nothing but the sum of the components of $\mbox{gr}_\nu R$ whose degree is in $q^{-1}(\gamma_1)$. Since we are dealing with graded modules whose homogeneous components are one-dimensional $K$-vector spaces, this module is finitely generated if and only if there exist finitely many homogeneous elements $e_1,\ldots ,e_r \in
\mbox{gr}_{{\mathcal F}(\gamma_1)}{\mathcal P}_{\gamma_1}/{\mathcal P}^+_{\gamma_1}$ such that each homogeneous element of $
\mbox{gr}_{{\mathcal F}(\gamma_1)}{\mathcal P}_{\gamma_1}/{\mathcal P}^+_{\gamma_1}
$ can be written $\overline xe_i$ with $\overline x\in R/P$; this is equivalent to an assertion on degrees, which is exactly that the $\overline F$-module $
F_{\gamma_1}$ is finitely generated.
\end{proof}

\section{An example with value group $\ZZ^2$ and non finitely generated semigroup.}\label{z2}

Let $K$ be a field. Let $r$ be a positive natural number such that
the characteristic of $K$ does not divide
$r$,
and let $a_{ij}$ for $0\le i\le 2$ and $1\le j\le r$
 be algebraically independent over $K$. Let $M$ be the field
$$
M=K\left(\{a_{ij}\mid 0\le i\le 2\mbox{ and }1\le j\le r\}\right).
$$
Let $G\cong\ZZ_r$ be the subgroup of the permutation group $S_r$ generated by the cycle $(1,2,\ldots,r)$. For $\sigma\in G$, define a $K$ automorphism of
$M$ by $\sigma(a_{ij})=a_{i,\sigma(j)}$.  Let $L=M^G$ be the fixed field of $G$.
We have an \'etale, Galois  morphism
$$
\Lambda:\PP_M^2\rightarrow \PP_L^2\cong \PP_M^2/G.
$$
Let $x_0,x_1,x_2$ be homogeneous coordinates on $\PP^2_K$.
Define $p_i=(a_{0i},a_{1i},a_{2i})\in\PP^2_M$ for $1\le i\le r$,
and let
$q=\Lambda(p_1)$. Since
$\{p_1,\ldots,p_r\}$ is an orbit of $\Lambda$ under the
action of $G$, we have that   $q=\Lambda(p_i)$ for $1\le i\le r$,
and ${\mathcal I}_q{\mathcal O}_{\PP^2_M}={\mathcal I}_{p_1}\cdots{\mathcal I}_{p_r}$.

Let $\overline\nu$ be the $m$-adic valuation of ${\mathcal O}_{\PP_L^2,q}$.
\begin{equation}\label{eq20}
\Gamma(\PP^2_L,{\mathcal O}_{\PP^2}(d)\otimes {\mathcal I}_q^n)
=\{F(x_0,x_1,x_2)\in\Gamma(\PP^2_L,{\mathcal O}_{\PP^2}(d))\mid
\overline\nu (F(1,\frac{x_1}{x_0},\frac{x_2}{x_0}))\ge n\}.
\end{equation}
Since $M$ is flat over $L$,
\begin{equation}\label{eq21}
\Gamma(\PP^2_L,{\mathcal O}_{\PP^2}(d)\otimes {\mathcal I}_q^n)\otimes_LM
=\Gamma(\PP^2_M,{\mathcal O}_{\PP^2}(d)\otimes (({\mathcal I}_{p_1}\cdots
{\mathcal I}_{p_r})^n)).
\end{equation}

Let $\nu_1$ be the $m$-adic valuation on $R=L[x_1,x_1,x_2]_{(x_0,x_1,x_2)}$.
The valuation ring $R_{\nu_1}$ of $\nu_1$ is $L[x_0,\frac{x_1}{x_0},\frac{x_2}{x_0}]_{(x_0)}$, with residue field
$R_{\nu_1}/ m_{\nu_1}\cong L(\frac{x_1}{x_0},\frac{x_2}{x_0})$. The inclusion of the valuation ring $R_{\overline\nu}$ of $\overline\nu$ into its quotient field
$L(\frac{x_1}{x_0},\frac{x_2}{x_0})$ determines a composite valuation
$\nu=\nu_1\circ \overline\nu$ on the field $L(x_0,x_1,x_2)$, which dominates $R$.
The valuation $\nu$ is rational and its value group of $\nu$ is $\ZZ\times\ZZ$ with the Lex order.

For $f\in L[x_0,x_1,x_2]$, let $L_f(x_0,x_1,x_2)$ be the leading form of
$f$. $L_f$ is a homogeneous form whose degree is the order $\mbox{ord}(f)$ of $f$ at the
homogeneous maximal ideal of $k[x_0,x_1,x_2]$.  The value of $f$ is
$$
\nu(f)=(\mbox{ord}(f),\overline\nu(L_f(1,\frac{x_1}{x_0},\frac{x_2}{x_0})))\in\NN\times\NN.
$$

For $(d,n)\in\NN\times\NN$, we have $L$ module isomorphisms
$$
{\mathcal P}_{(d,n)}\cong
\Gamma(\PP^2_L,{\mathcal O}_{\PP^2}(d)\otimes {\mathcal I}_q^n)
\bigoplus\left(\bigoplus_{m>d}\Gamma(\PP^2_L,{\mathcal O}_{\PP^2}(m))\right),
$$
and
\begin{equation}\label{eq22}
{\mathcal P}_{(d,n)}/{\mathcal P}_{(d,n)}^+\cong
\Gamma(\PP^2_L,{\mathcal O}_{\PP^2}(d)\otimes {\mathcal I}_q^n)/
\Gamma(\PP^2_L,{\mathcal O}_{\PP^2}(d)\otimes {\mathcal I}_q^{n+1}).
\end{equation}

\begin{Proposition}\label{Prop1}
Suppose that $r=s^2$ where $s\ge 4$ is a natural number.
  Then
\begin{equation}\label{eq23}
{\mathcal P}_{(d,n)}/{\mathcal P}_{(d,n)}^+=0 \mbox{ if $d\le ns$}
\end{equation}
and
if $s'$ is a real number such that $s'>s$, then there exist natural numbers
$d,n$ such that $d<ns'$ and
\begin{equation}\label{eq24}
{\mathcal P}_{(d,n)}/{\mathcal P}_{(d,n)}^+\ne 0.
\end{equation}
\end{Proposition}

\begin{proof}
We have that $\Gamma(\PP^2_M,{\mathcal O}_{\PP^2}(d)\otimes\left(({\mathcal I}_{p_1}\cdots {\mathcal I}_{p_r})^n)\right)=0$
if $d\le ns$ by (1) of Proposition 1 of \cite{N}.  Thus (\ref{eq23}) is true by
formulas (\ref{eq21}) and (\ref{eq22}).

Suppose that $s'$ is a real number greater than $s$.
 Using the Riemann Roch Theorem and Serre Duality on the blow up of $\PP^2_L$ at $q$, as in the proof of (2) of Proposition 1 of \cite{N}, we compute that
$$
\mbox{dim}_L(\Gamma(\PP^2_L,{\mathcal O}_{\PP^2}(d)\otimes {\mathcal I}_q^n)
\ge \frac{d(d+3)}{2}-r\frac{n(n+1)}{2}+1>0
$$
if
$$
\frac{d(d+3)}{2}\ge r\frac{n(n+1)}{2}.
$$
Fixing a rational number $\lambda$ with $s'>\lambda>s$, we see that we can find natural numbers $d$ and $m$ with $\frac{d}{m}=\lambda$ and
$$
\mbox{dim}_L(\Gamma(\PP^2_L,{\mathcal O}_{\PP^2}(d)\otimes {\mathcal I}_q^m)\ne 0.
$$
Let $n\ge m$ be the largest natural number such that
$$
\mbox{dim}_L(\Gamma(\PP^2_L,{\mathcal O}_{\PP^2}(d)\otimes {\mathcal I}_q^n))\ne 0.
$$
Then $d<ns'$ and ${\mathcal P}_{(d,n)}/{\mathcal P}_{(d,n)}^+\ne0$ by (\ref{eq22}). Thus (\ref{eq24}) holds.
\end{proof}

The following result follows from Proposition \ref{Prop1}

\begin{Proposition}\label{Prop2} Suppose that $s\ge 4$ is a natural number. Then there exists a field $M$,  a rational function field
$M(x_0,x_1,x_2)$ in 3 variables, and a rank 2 valuation $\nu$ of $M(x_0,x_1,x_2)$ with value group $\ZZ\times\ZZ$ with the Lex order,
which dominates the regular local ring
$$
R=M[x_0,x_1,x_2]_{(x_0,x_1,x_2)}
$$
such that
$$
\bigoplus_{(d,n)\in \NN\times \NN}{\mathcal P}_{(d,n)}/{\mathcal P}_{(d,n)}^+
$$
is not a finitely generated $R/m_R\cong M$ algebra.

The semigroup $\Gamma=\nu(R\setminus \{0\})$ is not finitely generated as a semigroup. Further,  the closed rational cone generated by $\Gamma$ in $\RR^2$ is the rational polyhedron
$$
\{(d,n)\in\RR^2\mid n\ge 0\mbox{ and }d\ge ns\}.
$$
\end{Proposition}

\section{semigroups of ordinal type $\omega^h$}\label{accu}

Suppose that $G$ is an ordered abelian group of finite rank $n$. Then $G$ is order isomorphic to a
subgroup of $\RR^{n}$, where $\RR^{n}$ has the lex order (see Proposition 2.10 \cite{Ab}). If $G$
is the value group of a valuation $\nu$ dominating a Noetherian local ring $R$, and $S$ is the
semigroup of values attained by $\nu$ on $R$, then it can be shown that $S$ has no accumulation
points in $\RR^n$.

In this section we prove that this property is held by any well ordered semigroup $S$ of ordinal
type $\le \omega^h$,  which is contained in the nonnegative part of $\RR^{n}$. The proof relies on
the following lemma. The heuristic idea of the proof of the lemma is  that  the semigroup generated
by a set with an accumulation point  has many accumulation points, accumulations of accumulation
points, and so on.

\begin{Lemma}\label{LemmaO1} Let  $A$ be a well ordered set, which is contained in the
nonnegative part of $\RR$. Suppose that $A$ has an accumulation point in $\RR$, for the Euclidean
topology. For $m$ a positive integer, let
$$
mA=\{x_1+x_2+\cdots+x_m\mid x_1,\ldots,x_m\in A\}.
$$
 Then $mA$ contains a well ordered subset of ordinal type $\omega^m$.\end{Lemma}

\begin{proof} Let $\overline\lambda$ be an accumulation point of $A$. Since $A$ is well ordered, there exist
$\lambda_i\in A$ for $i\in\NN$, such that $\lambda_i<\lambda_j$ for $i<j$ and $\lim_{i\mapsto
\infty}\lambda_i=\overline\lambda$. Let $T_1=\{\lambda_i\}_{i\in \NN}$. $T_1\subset A$ is well
ordered of ordinal type $\omega$.

For $m\ge 2$ we will construct a well ordered subset $T_m$ of $mA$ which has ordinal type
$\omega^m$.

For $a_1\in\NN$, choose $\sigma^m_1(a_1)\in\NN$ such that
$$
\lambda_{a_1}+\overline\lambda<\lambda_{a_1+1}+\lambda_{\sigma_1^m(a_1)}.
$$
Now for $a_2\in \NN$, choose $\sigma_2^m(a_1,a_2)\in\NN$ such that
$$
\lambda_{\sigma_1^m(a_1)+a_2}+\overline\lambda<\lambda_{\sigma_1^m(a_1)+a_2+1}+\lambda_{\sigma_2^m(a_1,a_2)}.
$$
We iterate for $2\le i\le m$ to choose for each $a_i\in \NN$, $\sigma_i^m(a_1,a_2,\ldots,a_i)\in\NN$
such that
$$
\lambda_{\sigma_{i-1}^m(a_1,\ldots,a_{i-1})+a_i}+\overline\lambda<
\lambda_{\sigma_{i-1}^m(a_1,\ldots,a_{i-1})+a_i+1}+\lambda_{\sigma_i^m(a_1,\ldots,a_i)}.
$$

Let $T_m$ be the well ordered subset of $mA$ which is the image of the order preserving inclusion
$\NN^m\rightarrow mA$ where $\NN^m$ has the lex order, defined by
$$
(a_1,\ldots,a_m)\mapsto \lambda_{a_1+1}+\left( \sum_{i=2}^m
\lambda_{\sigma^m_{i-1}(a_1,\ldots,a_{i-1})+a_i+1}\right).
$$
By its construction, $T_m$ has ordinal type $\omega^m$.
 \end{proof}

\begin{Theorem}\label{TheoremO2}
Let $S$ be a  well ordered subsemigroup of the nonnegative part of $\RR^n$ (with the lex order) for
some $n\in\NN$. Suppose that $S$ has ordinal type $\le \omega^h$ for some $h\in\NN$.
  Then $S$ does not have an accumulation point in $\RR^n$ for the Euclidean
topology.
\end{Theorem}

\begin{proof} We prove the theorem by induction on $n$. For $n=1$, the theorem follows from  Lemma
\ref{LemmaO1}. Suppose that the theorem is true for subsemigroups of $\RR^{n-1}$.

Suppose that $S\subset \RR^n$ has an accumulation point $\alpha$.

Let $\pi:\RR^n\rightarrow \RR^{n-1}$ be projection onto the first $n-1$ factors. For $x,y\in
\RR^n$, we have that $x\le y$ implies $\pi(x)\le\pi(y)$.  Thus the set $\pi(S)$ is a well ordered
semigroup, which  has ordinal type $\le \omega^h$ and is contained in the nonnegative part of
$\RR^{n-1}$. Let $\overline \alpha=\pi(\alpha)$. By the induction assumption, $\pi(S)$ has no
accumulation points. Thus $\overline \alpha\in\pi(S)$, and there exists an open neighborhood $U$ of
$\overline \alpha$ in $\RR^{n-1}$ such that $U\cap \pi(S)=\{\overline \alpha\}$. Thus $\pi^{-1}(U)\cap
S=\pi^{-1}(\{\overline \alpha\})\cap S$, and $\alpha$ is an accumulation point of $\pi^{-1}(\{\overline
\alpha\})\cap S$.

Projection on the last factor is a natural homeomorphism of $\pi^{-1}(\{\overline \alpha\})$ to $\RR$
which is order preserving. Let $A\subset \RR$ be the image of $\pi^{-1}(\{\overline \alpha\})\cap S$.
Since $A$ is a well ordered set which has an accumulation point, by Lemma \ref{LemmaO1}, $mA$ has a
subset which has ordinal type $\omega^m$ for all $m\ge 1$.  Now projection onto the last factor
identifies  a subset of $\pi^{-1}(\{m\overline \alpha\})\cap S$ with $mA$ for all $m\ge 1$. Thus the
ordinal type of $S$ is $\ge \omega^m$ for all $m\in\NN$, a contradiction.
\end{proof}

A variation of the proof of Theorem \ref{TheoremO2} proves the following corollary.

\begin{Corollary}\label{CorollaryO3}
Let $S$ be a  well ordered subsemigroup of the nonnegative part of $\RR^n$ (with the lex order) for
some $n\in\NN$. Suppose that $S$ has ordinal type $\le \omega^h$ for some $h\in\NN$.
  Then $S$ has ordinal type $\le \omega^n$.
\end{Corollary}

\begin{proof} We prove the theorem by induction on $n$.

We first establish the theorem for $n=1$. Suppose that  $S\subset\RR$ has ordinal type $>\omega$.
To the ordinal number $\omega$ there corresponds an element $a$ of $S$. There are then infinitely
many elements of $S$ in the closed interval $[0,a]$, so that $S$ has an accumulation point. By
Theorem \ref{TheoremO2}, this is impossible.

Now suppose that $S\subset \RR^n$ and the theorem is true for subsemigroups of $\RR^{n-1}$. By
Theorem \ref{TheoremO2}, $S$ has no accumulation points.

Let $\pi:\RR^n\rightarrow \RR^{n-1}$ be projection onto the first $n-1$ factors. The set $\pi(S)$
is a well ordered semigroup, which  has ordinal type $\le \omega^h$ and is contained in the
nonnegative part of $\RR^{n-1}$.  By the induction assumption, $\pi(S)$ has ordinal type $\le
\omega^{n-1}$. For $\overline x\in\pi(X)$, $\pi^{-1}(\{\overline x\})\cap S$ is a well ordered
subset of $\RR$ which has no accumulation points, and thus  has ordinal type $\le \omega$. Since
this is true for all $\overline x\in \pi(S)$, the ordinal type of $S$ is $\le
\omega\omega^{n-1}=\omega^n$.
\end{proof}

\section{Subsemigroups of $\QQ_+$}\label{ratrk1}
Let $S$ be a well ordered subsemigroup of $\QQ_+$. Let $(\gamma_i)_{i\in I}$ be its minimal system of generators. The set of the $\gamma_i$ may or may not be of ordinal type $\omega$.\par
For example let us choose two prime numbers $p,q$ and consider the positive rational numbers $\gamma_i=1-\frac{1}{p^i}$ for $1\leq i<\omega$ and $\gamma_i=2-\frac{1}{q^{i-\omega+1}}$ for $\omega \leq i<2\omega$. These numbers form a well ordered subset of $\QQ_+$ of ordinal type $2\omega$ and generate a certain semigroup $S_{p,q}$ which in turn is well ordered by a result of B.H. Neumann (see \cite{Ne} and \cite{Ri}, Theorem 3.4). Because of the way their denominators grow with $i$, the $\gamma_i$ are a minimal system of generators of $S_{p,q}$. Using the first $k$ prime numbers one can build in the same way well ordered semigroups in $\QQ_+$ with minimal systems of generators of ordinal type $k\omega$ for any $k<\omega$. \par\medskip\noindent
\textbf{Remark}:
Since we assume $S\subset \QQ_+$, if it is a semigroup of values for some valuation, that valuation is of rank one and by the result quoted in the Introduction, if the semigroup $S$ comes from a noetherian ring it is of ordinal type $\leq \omega$.  \par\noindent
 The semigroup $S_{p,q}$ is therefore an example of a well ordered subsemigroup of a totally ordered group of finite rank which cannot be realized as the semigroup of values of a noetherian ring. Note that by Corollary \ref{CorollaryO3} the ordinal type of $S_{p,q}$ is $\geq \omega^\omega$. One can ask whether equality holds.
\par\medskip\noindent
We shall from now on in this section consider only semigroups $S=\langle \gamma_1,\ldots ,\gamma_i,\ldots\rangle\subset \QQ_+$ whose minimal system of generators is of ordinal type $\leq\omega$.  If the $\gamma_i$ have a common denominator, $S$ is isomorphic to a subsemigroup of $\NN$, and finitely generated (for a short proof see \cite{Re}, Th. 83 p. 203).\par Let us therefore assume that there is no common denominator. Let us denote by $S_i$ the semigroup generated by $\gamma_1,\ldots ,\gamma_i$ and by $G_i$ the subgroup of $\QQ$ which it generates. We have $S=\bigcup_{i=1}^\infty S_i$. Set also $n_i=[G_i\colon G_{i-1}]$ for $i\geq 2$. It is convenient to set $n_1=1$. The products $\Pi_{i=1}^kn_i$ tend to infinity with $k$.\par
By definition we have $n_i\gamma_i\in G_{i-1}$ and the image of $\gamma_i$ is an element of order $n_i$ in $G_i/G_{i-1}$. For each $i\geq 1$ let $r_i$ be the positive rational number such that $r_i\gamma_1,r_i\gamma_2,\ldots , r_i\gamma_i$ generate the group $\ZZ$. Since the semigroup $r_{i-1}S_{i-1}$ generates $\ZZ$ as a group, it has a finite complement (see \cite{Re}, Th. 82) and some positive multiple of $r_{i-1}n_i\gamma_i$ is contained in it. Thus, we know that there exists a smallest integer $s_i$ such that $s_i\gamma_i\in S_{i-1}$ and that it is an integral multiple of $n_i$.\par\medskip\noindent
\begin{definition} Let $S$ be a well ordered subsemigroup of the semigroup of positive elements of a totally ordered abelian group of finite rank. Let $\gamma_1,\ldots ,\gamma_i,\ldots$ be a minimal system of generators of $S$ indexed by some ordinal $\alpha$ and for each ordinal $\beta<\alpha$ define $S_\beta$ to be the semigroup generated by the $(\gamma_i)_{i\leq \beta}$. Given an integer $d$ we say that $S$ has stable asymptotic embedding dimension $\leq d$ if for each $\beta<\alpha$ the semigroup $S_\beta$ is isomorphic to the semigroup of values which a valuation takes on an equicharacteristic noetherian local domain of embedding dimension $\leq d$. We say that $S$ has stable embedding dimension $\leq d$ if it is the semigroup of values of a valuation on an equicharacteristic noetherian local domain of embedding dimension $\leq d$.
\end{definition}\par\noindent
\begin{proposition}\label{planesemigr} Let $S=\langle \gamma_1,\ldots ,\gamma_i,\dots\rangle$ be a well ordered subsemigroup of $\QQ_+$ which is not isomorphic to $\NN$ and whose minimal system of generators is of ordinal type $\leq\omega$. The following are equivalent:\par\noindent
1) For each $i\geq 2$ we have $s_i=n_i$ and $\gamma_i>s_{i-1}\gamma_{i-1}$.\par\noindent
2) The stable asymptotic embedding dimension of $S$ is two.\par\noindent
3) The stable embedding dimension of $S$ is two.
\end{proposition}
\begin{proof} The conditions of 1) are known to be equivalent to the fact that each semigroup $S_i$ is the semigroup of values of the natural valuation of a plane branch, which is of embedding dimension 2 since $S\neq \NN$ (see \cite{Z}, Appendix, for characteristic zero, and \cite{C} for what is necessary to extend to positive characteristic). This shows that 1) is equivalent to 2). \par
Given a sequence of $\gamma_i$ satisfying 1), we can associate to it a sequence of key polynomials (SKP) as in (\cite{FJ}, Chapter 2, Definition 2.1) over any algebraically closed field $K$. That is, a sequence $P_0=x, P_1=y,\ldots ,P_i,\ldots$ of polynomials in $K[x,y]$ such that the conditions $\nu(P_i)=\gamma_i$ for all $i$ determine a unique valuation $\nu$ of the regular local ring $K[x,y]_{(x,y)}$ or, if the sequence of $\gamma_i$ is finite, of a one-dimensional quotient $K[x,y]_{(x,y)}/(Q(x,y))$, which is of embedding dimension 2 since $S\neq \NN$. The semigroup of values of $\nu$ is the semigroup generated by the $\gamma_i$ (see \textit{loc.cit.}, Theorem 2.28). So we see that 1) implies 3). Finally, if the semigroup $S$ comes from a noetherian local ring of embedding dimension 2, since for valuations of rank one the semigroup does not change under $m$-adic completion (see \cite{Te}, \S 5), we may assune that this ring corresponds either to a branch or to a two-dimensional complete equicharacteristic regular local local ring. If $S$ is the semigroup of values of a plane branch,  condition 1) is satisfied, as we have seen above and if it comes from a valuation $\nu$ of $K[[x,y]]$, by (\cite{FJ}, Theorem 2.29), there exists a SKP associated to $\nu$, and again condition 1) is satisfied. One should note that \cite{FJ} is written over the complex numbers, but the results of Chapter 2 are valid in any characteristic.\end{proof}\par\medskip\noindent
\textbf{Remark}: One may ask whether a subsemigroup of $\QQ_+$ with a minimal system of generators of ordinal type $\omega$ is always of bounded stable asymptotic embedding dimension or of bounded stable embedding dimension.
\section{The semigroups of noetherian rings are not always rationally finitely generated}
Using the results of \cite{FJ}, Chapter 2, one can check that for the semigroup $S=\langle \gamma_1,\ldots, \gamma_i,\ldots\rangle$ of a valuation of the ring $K[x,y]_{(x,y)}$ there always exist a finite set of generators $\gamma_1,\ldots \gamma_\ell $ which rationally generate the semigroup $S$ in the sense that for any generator $\gamma_j$ there is a positive integer $s_j$ such that $s_j\gamma_j\in \langle \gamma_1,\ldots, \gamma_\ell\rangle$. This fails for polynomial rings of dimension $\geq 3$, as is shown by the following example taken from \cite{Te}.\par\noindent 
Let us give
$\ZZ^2$ the lexicographic order and consider the field $K((t^{\ZZ^2_{\rm lex}}))$ endowed with the $t$-adic valuation with values in $\ZZ^2_{\rm lex}$. Let us denote by
$K[[t^{\ZZ^2_+}]]$ the corresponding valuation ring. Choose a 
sequence
of pairs of positive integers $(a_i,b_i)_{i\geq 3}$ and a sequence of elements $(\lambda_i\in K^*)_{i\geq 3}$ such that $b_{i+1}>b_i$, the series
$\sum_{i\geq 3}\lambda_i u_2^{b_i}$ is not algebraic over $K[u_2]$, and the ratios
$\frac{a_{i+1}-a_i}{b_{i+1}}$ are positive and increases strictly with $i$. Let $R_0$ be the
$K$-subalgebra of $K[[t^{\ZZ^2_+}]]$ generated by
$$u_1=t^{(0,1)}, u_2=t^{(1,0)},u_3=\sum_{i\geq 3}\lambda_i u_1^{-a_i}u_2^{b_i}.$$
There
cannot be an algebraic relation between $u_1,u_2,$ and $u_3$, so the ring
$R_0=K[u_1,u_2,u_3]$ is the polynomial ring in three variables. It inherits the $t$-adic 
valuation of
$K[[t^{\ZZ^2_+}]]$. One checks that this valuation extends to the localization
$R=K[u_1,u_2,u_3]_{(u_1,u_2,u_3)}$; it is a rational valuation of height two and rational rank
two.  Let us compute the semigroup $S$ of the values that it takes on $R$. We have
$\gamma_1=(0,1),\ \gamma_2=(1,0),\gamma_3=(b_3,-a_3)\in S$. Set
$S_3=\langle \gamma_1,\gamma_2,\gamma_3\rangle$. Then we have
$u_1^{a_3}u_3-\lambda_3 u_2^{b_3}=\sum_{i\geq 4}\lambda_i u_1^{a_3-a_i}u_2^{b_i}\in R$,  so that
$\gamma_4=(b_4,a_3-a_4)$ is in $S$. It is easy to deduce from our assumptions
that  {\it no
multiple of $\gamma_4$ is in $S_3$}, and that it is the smallest element of $S$
which is not in $S_3$. We set $u_4=u_1^{a_3}u_3-\lambda_3 u_2^{b_3}$, and continue in
the same manner: $u_1^{a_4-a_3}u_4-\lambda_4 u_2^{b_4}=u_5$, ...,
$u_1^{a_i-a_{i-1}}u_i-\lambda_i u_2^{b_i}=u_{i+1}$,... with
the generators $\gamma_i=\nu (u_i)=(b_i,a_{i-1}-a_i)$ for $i\geq 4$. Finally we have:
$$S=\langle \gamma_1,\gamma_2,\ldots ,\gamma_i ,\ldots\rangle ,$$ the initial forms of the $u_i$ constitute a minimal system of generators of
the graded $K$-algebra
$\hbox{\rm gr}_\nu R$, and the equations (setting $a_2=0$) $$u_1^{a_i-a_{i-1}}u_i-\lambda_i u_2^{b_i}=u_{i+1},\ \ i\geq 3$$
above  describe $\hat R^{(\nu )}$. In fact they even describe $R$; it is clear that from
them we can reconstruct the value of $u_3$ as a function of $u_1,u_2$ by (infinite)
elimination. The binomial equations for $\hbox{\rm gr}_\nu R$ are the 
$$U_1^{a_i-a_{i-1}}U_i-\lambda_i U_2^{b_i}=0,\ \ i\geq 3,$$ 
showing that all
the $U_i$ for $i\geq 3$ are  rationally dependent on $U_1,U_2$. From our assumption on the growth of the ratios we see moreover that {\it no multiple of}
$\gamma_i$ is in $S_{i-1}=\langle \gamma_1,\ldots ,\gamma_{i-1}\rangle$. In fact
$\gamma_i$ is outside of the cone with vertex $0$ generated by
$S_{i-1}$ in ${\mathbf R}^2$.\par\noindent

\section{A criterion for a series $\underline z$ to be transcendental}
In this section, we will use the following notation.
Let $K$ be a field, and $\Gamma$ a totally ordered abelian group.
Let $K((t^{\Gamma}))$ be the field of formal power series with
a well ordered set of exponents
in $\Gamma$, and coefficients in $K$. Let $\nu$ be the $t$-adic valuation of $K((t^{\Gamma}))$.
Suppose that $R\subset K((t^{\Gamma}))$ is
a local ring which is essentially of finite type over $K$ (a localization of
a finitely generated $K$-algebra). and that  $\nu$ dominates $R$.

Let $d=\mbox{dim}(R)$. Write $R=A_P$ where $A\subset R_{\nu}$ is of finite type
over $K$, and $R=A_P$ where $P$ is the center of $\nu$ on $A$.

By Noether's normalization theorem (Theorem 24, Section 7, Chapter VIII \cite{ZS}), there exist $x_1,\ldots,x_d\in A$, which are algebraically independent over $K$, such that $A$ is a finite module over the polynomial ring
$B=K[x_1,\ldots,x_d]$.  Thus there exist $b_1,\ldots,b_r\in A$ for some finite $r$,
such that $A=Bb_1+\cdots+Bb_r$.

Let $A[z]$ be a polynomial ring over $A$. For $n\in \NN$, define a finite dimensional $K$ vector space $D_n$ by
$$
D_n=\{f\in A[z]\mid f=f_1b_1+\cdots+f_rb_r\mbox{ where }f_1,\ldots,f_r\in K[x_1,\ldots,x_d,z]\mbox{ have total degree $\le n$}\}.
$$

\begin{Lemma}\label{Lemma0} Suppose that $w\in K((t^{\Gamma}))$ has positive value and $n\in \NN$. Then the set of values
$$
E_n=\{\nu(f(w))\mid f\in D_n\mbox{ and }f(w)\ne 0\}
$$
is finite.
\end{Lemma}

\begin{proof} For $\tau\in\Gamma_+$, let
$$
C_{\tau}=\{f\in D_{n}\mid \nu(f(w))\ge\tau\}.
$$
$C_{\tau}$ is a $K$ subspace of $D_n$. Since $C_{\tau_1}\subset C_{\tau_2}$ if $\tau_2\le\tau_1$, $E_n$ must be a finite set.
\end{proof}

\begin{Lemma}\label{Lemma1}
Suppose that $w\in K((t^{\Gamma}))$ has positive value. Let
$$
\tau=\mbox{max}\{\nu(f(w))\mid f\in D_n\mbox{ and }f(w)\ne 0\}.
$$
Choose $\lambda\in\Gamma$ such that $\lambda>\tau$ and $h\in K((t^{\Gamma}))$ such that $\nu(h)=\lambda$.

Suppose that $0\ne f\in D_n$. Then $f(w+h)\ne 0$.
\end{Lemma}

\begin{proof} Suppose that $0\ne f\in D_n$.
Let $m=\mbox{deg}_z(f)$. We have  $0<m\le n$ (the case $m=0$ is trivial).
Write
$$
f=a_mz^m+a_{m-1}z^{m-1}+\cdots+a_0
$$
where $a_m\ne 0$ and each $a_i$ has an expression
$$
a_i=c_{i1}b_1+\cdots+c_{ir}b_r
$$
where $c_{ij}\in K[x_1,\ldots,x_d]$ is a polynomial of degree $\le n$
for all $i,j$. Subsituting $w+h$ for $z$, we have
$$
f(w+h)=h^md_m(w)+h^{m-1}d_{m-1}(w)+\cdots +d_0(w)
$$
where $d_i(z)\in D_n$ for all $i$ and $d_m(z)=a_m$,
 so that $h^md_m(w)\ne 0$.

Suppose that $i<j$, $h^id_i(w)\ne 0$, $h^jd_j(w)\ne 0$ and
$\nu(h^id_i(w))=\nu(h^jd_j(w))$. Then
$$
i\lambda+\nu(d_i(w))=j\lambda+\nu(d_j(w)),
$$
which yields
$$
(j-i)\lambda=\nu(d_i(w))-\nu(d_j(w)).
$$
But
$$
\nu(d_i(w))-\nu(d_j(w))\le \tau<\lambda\le (j-i)\lambda,
$$
a contradiction. Thus all nonzero terms $h^id_i(w)$ of $f(w+h)$ have distinct values.
Since at least one of these terms was shown to be non zero, $f(w+h)$
has finite value, so $f(w+h)\ne 0$.
\end{proof}

\begin{Theorem}\label{Theorem1}
Suppose that $K$ is a field, $\Gamma$ is a totally ordered abelian group,
$R\subset K((t^{\Gamma}))$ is a local ring which is essentially of finite type over $K$ and which is dominated by the $t$-adic valuation $\nu$ of $K((t^{\Gamma}))$.

Suppose that $z_i\in K((t^{\Gamma}))$ are defined as follows:

Let
\begin{equation}\label{eq4}
\tau_1=\mbox{max}\{\nu(f)\mid f\in D_1\mbox{ and }f(0)\ne 0\}.
\end{equation}
Choose $\alpha_1\in\Gamma_+$   with $\alpha_1>\tau_1$ and $h_1\in K((t^{\Gamma}))$ such that $\nu(h_1)=\alpha_1$.
Set $z_1=h_1$.

Inductively define $\alpha_i\in \Gamma_+$, $h_i\in K((t^{\Gamma}))$ and $z_i=z_{i-1}+h_i$ with $\nu(h_i)=\alpha_i$ for $2\le i$ so that if

\begin{equation}\label{eq3}
\tau_{i}=\mbox{max}\{\nu(f(z_{i-1}))\mid f\in D_i
\mbox{ and }f(z_{i-1})\ne 0\},
\end{equation}
then
$$
\alpha_{i}>\tau_{i}.
$$

Then
$\underline z=\lim_{i\mapsto \infty}z_i\in K((t^{\Gamma}))$ is transcendental over the quotient field $L$ of $R$.
\end{Theorem}

\begin{proof}
Since $\{\alpha_i\}$ is an increasing sequence in $\Gamma$ with $\nu(h_j-h_i)=\alpha_{i+1}$ for $j>i$, we have that the limit $\underline z=\lim_{i\mapsto \infty}z_i$ exists in $K((t^{\Gamma}))$.

 Assume that $\underline z$ is not
transcendental over $L$. Then there exists a nonzero polynomial $g(z)\in A[z]$ such that
$g(\underline z)=0$. Let
$m=\mbox{deg}_z(g)\ge 1$.
Expand
$$
g(z)=a_mz^m+a_{m-1}z^{m-1}+\cdots+a_0
$$
with $a_i\in A$ for all $i$. Each $a_i$ has an expansion
$$
a_i=c_{i1}b_1+\cdots+c_{ir}b_r
$$
where $c_{ij}\in K[x_1,\ldots,x_d]$ for all $i,j$. Let $n\in\NN$ be such that $n\ge\mbox{deg}(c_{ij})$ for $1\le i\le m$ and $1\le j\le r$, and $n\ge m$.

By our construction, we have
$$
z_n=z_{n-1}+h_n\mbox{ with }\nu(h_n)=\alpha_n>\tau_n.
$$
 By Lemma \ref{Lemma1}, $g(z_n)\ne 0$. In $K((t^{\Gamma}))$, we compute
$$
g(\underline z)=g(z_n+\underline z-z_n)
=g(z_n)+(\underline z-z_n)e
$$
with $\nu(e)\ge 0$.
We have
$$
\nu(\underline z-z_n)= \nu(h_{n+1})=\alpha_{n+1}>
\tau_{n+1}\ge
\nu(g(z_n)).
$$
Thus $\nu(g(\underline z))\le \tau_{n+1}<\infty$, so that $g(\underline z)\ne 0$,
a contradiction.
\end{proof}

\begin{Corollary}
Suppose that $K$ is a field, $\Gamma$ is a totally ordered abelian group,
$R\subset K((t^{\Gamma}))$ is a local ring which is essentially of finite type over $K$ and which is dominated by the $t$-adic valuation $\nu$ of $K((t^{\Gamma}))$. Then there exists $\underline z\in K((t^{\Gamma}))$ such that
$\underline z$ is transcendental over the quotient field of $R$.
\end{Corollary}

\begin{Remark} The conclusions of the theorem may fail if $R$ is
Noetherian, but not essentially of finite type over $K$. A simple example is
$\Gamma=\ZZ$ and $R=K[[t]]$, since $K((t^{\Gamma}))$ is the quotient field of $R$.
\end{Remark}

\section{An example where all $\mbox{\rm gr}_{{\mathcal F}(\gamma_1)} {\mathcal P}_{\gamma_1}/{\mathcal P}_{\gamma_1}^+$ are not finitely generated $\mbox{\rm gr}_{\overline\nu}R/P$ modules}\label{z7}

Let $K$ be an algebraically closed field. Let $K(x,y)$ be a two dimensional
rational function field over $K$. Let $A=K[x,y]_{(x,y)}$.
In Example 3 of
Chapter VI, Section 18 of \cite{ZS} a construction is given of a valuation
$\overline \nu$ of $K(x,y)$ which dominates $A$, and such that the value group
of $\overline \nu$ is $\QQ$.
 There is an embedding $K(x,y)\subset K((t^{\QQ}))$ of $K$ algebras,
where $K((t^{\QQ}))$ is the formal power series field with a well ordered set of exponents in $\QQ$, and coefficients in $K$,
by Theorem 6 of \cite{Kap1}, since $K$ is algebraically closed and $\QQ$ is divisible.

Let $\QQ_+$ denote the positive rational numbers.

Let
\begin{equation}\label{eq31}
F_0=\{\overline\nu(f)\mid f\in A\text{ and }f\ne 0\}
\end{equation}
 be the semigroup of the valuation $\overline\nu$ on $A$.

We will use the criterion of Theorem \ref{Theorem1} to construct a limit
$\underline z =\lim_{i\rightarrow \infty}z_i$ in $K((t^{\QQ})$ which is
transcendental over the quotient field of $A$.

Let $z$ be a transcendental element over $K[x,y]$ and let
$$
D_i=\{f\in K[x,y,z]\mid \mbox{ the total degree of $f$ is $\le i$}\}.
$$

In our construction, we inductively define $z_i\in K((t^{\QQ}))$. Then for $n\in\NN$, we may
define $F_0=\overline\nu(A\setminus \{0\})$
modules $M_i^n$ by
$$
M_i^n=\{\nu(a_0+a_1z_i+\cdots+a_nz_i^n)\mid a_0,\ldots,a_n\in K[x,y]\}.
$$

Let
$$
A_i=A[z_i]_{(x,y,z_i)}.
$$
The local ring $A_i$ is dominated by $\overline\nu$, so the semigroup $\Gamma_i$
of $\overline\nu$ on $A_i$ is topologically discrete.  It follows that there are arbitrarily large elements of $\QQ_+$ which are not in $\Gamma_i$.

We first choose $\lambda_1\in\QQ_+$ such that $\lambda_1\not\in F_0$
and $\lambda_1>\tau_1$ where $\tau_1$ is defined by (\ref{eq4}).
Set $\alpha_1=\lambda_1$. Choose $f_1,g_1\in K[x,y]$ such that
$$
\overline\nu\left(\frac{f_1}{g_1}\right)=\alpha_1.
$$

We then inductively construct $\lambda_i\in\QQ_+$, $\alpha_i\in \QQ_+$,  $f_i,g_i\in K[x,y]$ and $\tau_i\in\QQ_+$ such that
$\lambda_{i}\not\in \Gamma_{i-1}$, $\tau_i$ is defined by (\ref{eq3}),
and
$$
\lambda_{i}>\mbox{max}\{\lambda_{i-1}+\overline\nu(g_1\cdots g_{i-1}),\tau_i+\overline\nu(g_1\cdots g_{i-1})\}.
$$
We define
$$
\alpha_{i}= \lambda_{i}-\overline\nu(g_1\cdots g_{i-1})
$$
and choose $f_i,g_i\in K[x,y]$ so that
$$
\overline\nu\left(\frac{f_i}{g_i}\right)=\alpha_i.
$$
Let
$$
z_i=z_{i-1}+\frac{f_i}{g_i}.
$$
The resulting series $\underline z=\lim_{i\rightarrow \infty}z_i$ is transcendental over $K(x,y)$
by Theorem \ref{Theorem1}, since
$\alpha_i>\tau_i$ for all $i>1$.

\begin{Lemma}\label{Lemma20}
With the above notation, for $n\in\NN$,
define a $F_0$ module
$$
T^n=\{\overline\nu(a_0+a_1\underline z+\cdots+a_n\underline z^n)\mid a_0,a_1,\cdots,a_n\in K[x,y]\}.
$$
Then for all $n>0$, $T_i^n$ is not  finitely generated as a $F_0$ module.

Let $B=A[\underline z]_{(x,y,\underline z)}$. $B\subset K((t^{\QQ}))$ is
a three dimensional local ring which is dominated by the $t$-adic valuation $\overline\nu$ of
$K((t^{\QQ}))$. Let
$T^{\infty}=\overline\nu(B\setminus\{0\})$ be the semigroup of the valuation $\overline\nu$ on $B$.
Then $T^{\infty}$ is not  a finitely generated $F_0$ module.
\end{Lemma}

\begin{proof}

Suppose that $n\ge 1$.
We will show that $T^n$ is not finitely generated as a $F_0$ module.

We compute
$$
\overline\nu(\underline z)=\overline\nu\left(\frac{f_1}{g_1}\right)=\lambda_1,
$$
and for $i\ge 2$,
$$
\begin{array}{l}
\overline\nu\left(g_1\cdots g_{i-1}\underline z - (f_1g_2\cdots g_{i-1}+f_2g_1g_3\cdots g_{i-1}+\cdots + f_{i-1}g_1\cdots g_{i-2})\right)\\
=\overline\nu\left(g_1\cdots g_{i-1}\frac{f_i}{g_i}\right)\\
=\overline\nu\left(\frac{f_i}{g_i}\right)+\overline\nu(g_1\cdots g_{i-1})\\
=\lambda_i.
\end{array}
$$
Thus $\lambda_i\in T^n$ for all $i$.

For $n\in\NN$,
define a $F_0$ module
$$
T^n=\{\overline\nu(a_0+a_1\underline z+\cdots+a_n\underline z^n)\mid a_0,a_1,\cdots,a_n\in K[x,y]\}.
$$

$M_i^n$ is a subset of $\QQ$ for all $i$. We compare the intersections of
the $M_i^n$ with various intervals $[0,\sigma)$ in $\QQ$.
Since $\overline\nu(z_1)=\alpha_1$, we have that
$$
M_1^n\cap [0,\alpha_1)=   F_0\cap [0,\alpha_1),
$$
and since
$$
z_i=z_{i-1}+\frac{f_i}{g_i}
$$
 with
$$
\overline\nu\left(\frac{f_i}{g_i}\right)=\alpha_i,
$$
we have that
$$
M_i^n\cap [0,\alpha_i)=M^n_{i-1}\cap [0,\alpha_i)
$$
for all $i\ge 2$.

We further have that
\begin{equation}\label{eq2}
T^n\cap [0,\alpha_i)=M^n_i\cap [0,\alpha_i)
\end{equation}
for $i\ge 1$.

Suppose that $n\ge 1$ and $T^n$ is a finitely generated $F_0$ module. We will derive
a contradiction. With this assumption, there exist $x_1,\ldots,x_m\in T^n$
such that every element $v\in T^n$ has an expression $v=y+x_j$ for some
$y\in F_0$, and for some $x_j$ with $1\le j\le m$.

There exists a positive integer $l$ such that
$\nu(x_j)<\alpha_l$ for $1\le j\le l$. Thus $x_1\ldots,x_m\in M_l^n$ by
(\ref{eq2}). It follows  that $T^n\subset M_l^n$ since $M_l^n$ is a $F_0$ module.
But $\lambda_{n+1}\not\in M_l^n$ by our construction, as $M_l^n\subset \Gamma_l$, which gives a contradiction,
since we have shown that $\lambda_{l+1}\in T^n$.

The proof that $T^{\infty}=\bigcup_{n=0}^{\infty} T^n$ is not a finitely generated
$F_0$ module is similar.
\end{proof}

\begin{Proposition}\label{Proposition3} Suppose that $K$ is an algebraically closed field, and $K(x,y,z)$ is a rational function field in 3 variables.
Then there exists a rank 2 valuation $\nu=\nu_1\circ\overline\nu$ of $K[x,y,z]$ with value group $\ZZ\times\QQ$ with the Lex order, which dominates the
regular local ring
$$
R=K[x,y,z]_{(x,y,z)},
$$
with $R/P\cong K[x,y]_{(x,y)}$ where $P$ is the center of $\nu_1$ on $R$,
such that the associated graded module
$$
\mbox{gr}_{{\mathcal F}(n)} {\mathcal P}_{n}/{\mathcal P}_n^+
$$
is not a finitely generated $\mbox{gr}_{\overline\nu}R/P$ module
for all positive integers $n$.

\end{Proposition}

\begin{proof}
Since the $\underline z$ we have just constructed is transcendental over $K(x,y)$, the association $z\rightarrow \underline z$ defines
 an embedding of $K$ algebras
$K(x,y,z)\rightarrow K((t^{\QQ}))$ which extends our embedding
$K(x,y)\rightarrow K((t^{\QQ}))$.
We identify $\overline\nu$ with the induced valuation on $K(x,y,z)$, which by our
construction has value group $\QQ$ and residue field $K$.

Let $R$ be the localization
$$
R=K[x,y,u,v]_{(x,y,u,v)}
$$
of a polynomial ring in four variables.
Let $\nu_1$ be the $(u,v)$-adic
valuation of $R$. The valuation ring of the discrete, rank 1 valuation $\nu_1$ is $$
R_{\nu_1}=K\left[x,y,u,\frac{v}{u}\right]_{(u)}.
$$
The residue field of $R_{\nu_1}$ is the rational function field in three variables
$$
R_{\nu_1}/m_{\nu_1}=K(x,y,z)
$$
where $z=\frac{v}{u}$.
Let $\nu$ be the composite valuation $\nu_1\circ \overline\nu$ on $K(x,y,u,v)$.

For $i\in \NN$,  we have
$$
F_i=\{\nu(f)\mid f\in R\mbox{ and }\nu_1(f)=i\}.
$$
$F_0$ is the semigroup $F_0=\overline\nu(R/P\setminus\{0\})$, and $F_i$ are $F_0$ modules for all $i$.

We have that $m_{\nu_1}\cap R=(u,v)$, so that
$F_0=\Gamma_0$, the semigroup of (\ref{eq31}).  From our construction of $\overline\nu$, we have that
$F_n$ is isomorphic to $T^n$ as a $\Gamma_0$ module.
Thus for all $n\ge 1$, $F_n$ is not finitely generated as a $F_0$ module.

By Proposition \ref{Theorem5},
$$
\mbox{gr}_{{\mathcal F}(n)} {\mathcal P}_{n}/{\mathcal P}_n^+
$$
is not a finitely generated $\mbox{gr}_{\overline\nu}R/P$ module
for all positive integers $n$.

\end{proof}

\begin{Remark} In the example of Proposition \ref{Proposition3},
the residue field $R_{\nu_1}/m_{\nu_1}$ is transcendental over the quotient field of $R/P$, a fact which is used in the construction. In Proposition
\ref{Prop30} of the following section, an example is given where $R_{\nu_1}/m_{\nu_1}$ is equal to the quotient field of $R/P$.
\end{Remark}

\begin{Remark} We can easily construct a series $\underline z\in K((t^{\QQ}))$
such that the modules $T^n$ and $T^{\infty}$ of the conclusions of Lemma
\ref{Lemma20} are all finitely generated $\Gamma_0$ modules, and thus the
modules
$$
\mbox{gr}_{{\mathcal F}(n)} {\mathcal P}_{n}/{\mathcal P}_n^+
$$
are  finitely generated $\mbox{gr}_{\overline\nu}R/P$ module
for all positive integers $n$. To make the construction, just take $\underline z\in K[[x,y]]$ to be any transcendental series.
\end{Remark}

\section{An example with no residue field extension}\label{z8}

Suppose that $K$ is an algebraically closed field, and $A=K[x,y]_{(x,y)}$. We will define a
valuation $\overline\nu$ on $L=K(x,y)$ which dominates $A$, with  value group $\bigcup_{i=0}^{\infty}\frac{1}{2^i}\ZZ$.

Define $\overline\beta_0=1$ and $\overline\beta_{i+1}=2\overline\beta_i+\frac{1}{2^{i+1}}$ for
$i\ge 0$.
We have
$$
\overline\beta_i=\frac{1}{3}(2^{i+2}-\frac{1}{2^i})
$$
for $i\ge 0$, and
$$
2\overline\beta_i=\overline\beta_{i-1}+2^{i+1}\overline\beta_0
$$
for $\ge 1$.

Define groups
$$
\Gamma_i=\sum_{j=0}^i\ZZ\overline\beta_j=\frac{1}{2^i}\ZZ
$$
for $i\ge 0$.

For $i\ge 1$, let
$$
x_i=\frac{2^{i+1}\overline \beta_{i-1}+1}{2^{i-1}},
$$
and let $m_i=2$.

Since $2^{i+1}\overline\beta_{i-1}$ is an even integer for all $i$, $x_i$ has order $m_i=2$ in
$\Gamma_{i-1}/m_i\Gamma_{i-1}$ and $\overline\beta_i=\frac{x_i}{m_i}$.

By our construction, for $i\ge 1$, $\overline\beta_{i+1}>m_i\overline\beta_i$. By the
irreducibility criterion of \cite{CM-S}, Remark 7.17 \cite{CP}, or Theorem 2.22 \cite{FJ},  there exists a valuation $\overline\nu$ of $L$
dominating $A$ and
 a (minimal) generating sequence $P_0,P_1,\cdots,P_i,\cdots$
for $\overline\nu$ in $K[x,y]$ of the form
$$
\begin{array}{l}
P_0=x\\
P_1=y\\
P_2=y^{2}-x^{5}\\
P_3=P_2^2-x^8y\\
\vdots\\
P_{i+1}=P_i^{2}-P_0^{2^{i+1}}P_{i-1}\\
\vdots
\end{array}
$$
 with $\overline\beta_i=\nu(P_i)$ for all $i$.  The semigroup $M_0$ of $\overline\nu$ on $A$ is
$$
M_0=\sum_{i=0}^{\infty}\NN\overline\beta_i.
$$
Let $z=\frac{y}{x}\in L$, and define for $n\in \NN$,
$$
W_n=\{a_0+a_1z+\cdots+a_nz^n\mid a_0,\ldots,a_n\in K[x,y]\}.
$$
Define  $M_0$ modules $M_n$ by
$$
M_n=\{\overline\nu(f)\mid 0\ne f\in W_n\}.
$$
\begin{Lemma}\label{Lemma12} For all $i\ge 0$. we have an expression
$P_i=x^ih_i$ with $h_i\in W_i$.
\end{Lemma}
\begin{proof}  The statement is clear for $i=0$ and $i=1$.
Suppose, by induction, that the statement is true for $j\le i$, so that $P_j=x^jh_j$ with $h_j\in
W_j$ for $j\le i$. Write
$$
P_{i+1}=P_i^2-P_0^{2^{i+1}}P_{i-1}.
$$
$$
P_i^2=x^{i+1}\left[\left(\frac{P_i}{x^i}\right)\left(\frac{P_i}{x^i}\right)x^{i-1}\right]
$$
with
$$
\left(\frac{P_i}{x^i}\right)\left(\frac{P_i}{x^i}\right)x^{i-1}\in W_{2i-(i-1)}=W_{i+1}.
$$
Further,
$$
P_0^{2^{i+1}}P_{i-1}=x^{2^{i+1}}x^{i-1}h_{i-1}\in x^{i+1}W_{i-1}\subset x^{i+1}W_{i+1}.
$$
Thus $P_{i+1}=x^{i+1}h_{i+1}$ with $h_{i+1}\in W_{i+1}$.
\end{proof}

For $n\in \NN$, let
$$
U_n=\{\lambda\in M_n\mid \lambda=\sum_{j=0}^{n-1}l_j\overline\beta_j\mbox{ for some }l_i\in\ZZ\}.
$$

\begin{Lemma}\label{Lemma10}
For $n\ge 1$,
$$
M_n=U_n\bigcup\left(\bigcup_{j\ge n}((\overline\beta_j-n)+M_0)\right).
$$
\end{Lemma}
\begin{proof}
$\frac{P_j}{x^j}\in W_j$ for all $j\ge 0$ implies
$$
\frac{P_j}{x^n}=\left(\frac{P_j}{x^j}\right)x^{j-n}\in W_n
$$
for $j\ge n$. Thus $\overline\beta_j-n\in M_n$ for $j\ge n$.

Now suppose that $\lambda\in M_n$. $\lambda=\overline\nu(a_0+a_1z+\cdots+a_nz^n)$ for some
$a_0,\ldots,a_n\in K[x,y]$.   Set $\tau=\overline\nu(a_0x^n+a_1x^{n-1}y+\cdots+a_ny^n)\in M_0$. We have
$\lambda=\tau-n$. $\tau=\sum l_j\overline\beta_j=\overline\nu(\prod P_j^{l_j})$ for some $l_j\in \NN$.
Suppose $l_k\ne 0$ for some $k\ge n$. Then
$$
\lambda=\overline\nu\left(\left(\prod_{j\ne k}P_j^{l_j}\right)P_k^{l_k-1}\frac{P_k}{x^n}\right)=
(\overline\beta_k-n)+
\overline\nu\left(\left(\prod_{j\ne k}P_j^{l_j}\right)P_k^{l_k-1}\right).
$$
Now suppose that $l_k=0$ for $k<n$. Then
$$
\lambda=\sum_{j=1}^{n-1}l_j\overline\beta_j+(l_0-n)\overline\beta_0\in U_n.
$$
\end{proof}

\begin{Lemma}\label{Lemma11} Suppose that $n\ge 1$. Then $M_n$ is not a finitely generated $M_0$
module.

Let $B=K[x,\frac{y}{x}]_{(x,\frac{y}{x})}$. B is a regular local ring which
birational dominates $A=K[x,y]_{(x,y)}$. Let $M_{\infty}=\overline\nu(B\setminus\{0\})$.
Then $M_{\infty}$ is not a finitely generated module
over $\Gamma_0=\overline\nu(A\setminus\{0\})$.
\end{Lemma}

\begin{proof}
For $i\ge 0$, define $\Psi_i$ to be the $M_0$ module generated by $U_n$ and $\{
\overline\beta_j-n\mid i\ge j\ge n \}$.
For $i\ge n$ we have
$$
\Psi_i = U_n\bigcup \left(
\bigcup_{i\ge j\ge n}((\overline\beta_j-n)+M_0)\right)\subset
\frac{1}{2^i}\NN.
$$
Thus $\overline\beta_{i+1}-1\not\in\Psi_i$ for $i\ge n$.

We will now show that $M_n$ is not a finitely generated $M_0$ module.

Suppose that $M_n$ is finitely generated as an $M_0$ module. Then $M_n$ is generated by a set
$a_1,\ldots,a_r,b_1,\ldots,b_s$, where
$$
a_i=(\overline\beta_{\sigma(i)}-n)+\lambda_i
$$
with $\sigma(i)\ge n$ and $\lambda_i\in M_0$ for $1\le i\le r$ and $b_i\in U_n$ for $1\le i\le s$.

Let $m=\mbox{max}\{\sigma(i)\}$. Then $M_1\subset \Psi_m$, which is impossible since
$\overline\beta_{m+1}-n\not\in \Psi_m$.

The proof that $M_{\infty}$ is not a finitely generated $\Gamma_0$ module is similar.
\end{proof}

\begin{Proposition}\label{Prop30}
Suppose that $K$ is an algebraically closed field, and $K(x,y,u,v)$ is a rational function field in 4 variables.

Suppose that $\Gamma$ is a totally  ordered Abelian group  and $\alpha\in\Gamma$ is such that 1 and $\alpha$ are
rationally independent and $1<\alpha$.

Then there exists a rank 2 valuation $\nu=\nu_1\circ\overline\nu$ of $K(x,y,u,v)$ with value group
$$
\left(\ZZ+\alpha\ZZ\right)\times\left(\bigcup_{i=0}^{\infty}\frac{1}{2^i}\ZZ\right)
$$
in the Lex order, which dominates the regular local ring
$$
R=K[x,y,u,v]_{(x,y,u,v)}
$$
such that
\begin{enumerate}
\item $R_{\nu_1}/m_{\nu_1}\cong R/P\cong K[x,y]_{(x,y)}$ where $P$ is the center of $\nu_1$ on $R$.
\item  The associated graded module
$$
\mbox{gr}_{{\mathcal F}(n)}{\mathcal P}_n/{\mathcal P}^+_n
$$
is not  a  finitely generated  $\mbox{gr}_{\overline\nu}R/P$ module for  $n\in\NN$ a positive integer.
\item  The associated graded module
$$
\mbox{gr}_{{\mathcal F}(n\alpha)}{\mathcal P}_{n\alpha}/{\mathcal P}^+_{n\alpha}
$$
is a  finitely generated  $\mbox{gr}_{\overline\nu}R/P$ module for  $n\in\NN$.
\end{enumerate}
\end{Proposition}

\begin{proof}

We use the notation developed earlier in this section.
Define a valuation $\nu_1$ on the rational function field $L(u,v)$ in 2 variables  by the embedding of $L$ algebras
$$
L(u,v)\rightarrow L((t^{\Gamma}))
$$
induced by
$$
u\mapsto t, v\mapsto v(t)=\frac{y}{x}t+t^{\alpha}.
$$
Let $\nu=\nu_1\circ\overline\nu$ be the composite valuation on $K(x,y,u,v)$.
$\nu$ dominates $R=K[x,y,u,v]_{(x,y,u,v)}$.
The center of $\nu_1$ on $R$ is the prime ideal $P=(u,v)$. We have $L=(R/P)_{P}=R_{\nu_1}/m_{\nu_1}$ and $K=R_{\nu}/m_{\nu}$.

For $i\in\NN$,
$$
F_i=\{\nu(f)\mid f\in R\mbox{ and }\nu_1(f)=i\}.
$$
Suppose that $f\in K[x,y,u,v]$. Expand
$$
f=\sum a_{ij}u^iv^j
$$
with $a_{ij}\in K[x,y]$. We have
$$
f(t,v(t))=a_{00}+(a_{10}+a_{01}\frac{y}{x})t+\mbox{ higher order terms in $t$}.
$$
We see that $\nu_1(f)=0$ if and only if $a_{00}\ne 0$. Thus $F_0\cong M_0$
as semigroups.

$\nu_1(f)=1$ if and only  if $a_{00}=0$ and $a_{10}+a_{01}\frac{y}{x}\ne 0$. Thus $F_1\cong M_1$ as $F_0$ modules, so that $F_1$ is not finitely generated as an $F_0$ module.

To see that $F_n\cong M_n$ as an $F_0$ module for all $n\ge 0$, we expand

$$
\begin{array}{lll}
f(t,v(t))&=&\sum_{k=0}^{\infty}\sum_{i+j=k}a_{ij}t^iv(t)^j\\
&=& \sum_{k=0}^{\infty}\sum_{j=0}^k\phi_{jk}t^{(k-j)+j\alpha}
\end{array}
$$
where
$$
\phi_{jk}=\sum_{i=j}^ka_{k-i,i}\binom{i}{j}\left(\frac{y}{x}\right)^{i-j}.
$$
We have $(k_1-j_1)+j_1\alpha=(k_2-j_2)+j_2\alpha$ implies $j_1=j_2$ and $k_1=k_2$. Thus
$$
\nu_1(f)=\mbox{min}\{(k-j)+j\alpha\mid \phi_{jk}\ne 0\}.
$$
We further have for fixed $k$, $j_1<j_2$ implies
\begin{equation}\label{eq15}
(k-j_1)+j_1\alpha<(k-j_2)+j_2\alpha.
\end{equation}

Suppose that $\nu_1(f)=n\in\NN$. Then
$$
f(t,v(t))=\phi_{0n}t^n+\mbox{ higher order terms in $t$}
$$
with
$$
\phi_{0n}=a_{0n}+a_{n-1,1}\frac{y}{x}+\cdots+a_{0n}\left(\frac{y}{x}\right)^n\ne 0.
$$
Further, by (\ref{eq15}), we see that for $n\in\NN$ and $a_{n0},a_{n-1,1},\ldots,a_{0n}\in K[x,y]$,
$$
\nu_1(a_{n0}u^n+a_{n-1,1}u^{n-1}v+\cdots+a_{0n}v^n)=n
$$
if and only if
$$
a_{n0}+a_{n-1,1}\frac{y}{x}+\cdots+a_{0n}\left(\frac{y}{x}\right)^n\ne 0.
$$
Thus for $n\in\NN$,
$$
F_n\cong \{\overline\nu(h)\mid h\in W_n\mbox{ and }h\ne 0\}
$$
which is isomorphic to $M_n$ as an $M_0$ module. Since $M_n$ is not finitely generated as an $M_0$ module, we obtain (2) of the conclusions of the Proposition
from Proposition \ref{Theorem5}.

Let us now consider the polynomial $$
\Phi_{jk}(W)=\sum_{i=j}^ka_{k-i,i}\binom{i}{j}W^{i-j},$$
and remark that we have the equalities
$$\frac{\partial^j\Phi_{0k}(W)}{\partial W^j}=j! \Phi_{jk}(W).$$
In order for the series $f(t,v(t))$ to be of order $n\alpha$, all the $\phi_{jk}$ must be zero for $(k-j)+j\alpha<n\alpha$,
while $\phi_{nn}=a_{0n}$ must be non zero. In particular,
$\phi_{jn}$ must be zero for $j<n$. \par
In view of the equalities we have just seen, the $n$ conditions $\phi_{jn}=0, \ 0\leq j\leq n-1$ are equivalent to the fact that $\frac{y}{x}$ is a root of order $n$ of the polynomial $\Phi_{0n}(W)$, so that $\Phi_{0n}(W)=a_{0n}(W-\frac{y}{x})^n$. From this it follows that the $a_{n-j,j}$ for $j<n$ are determined and the only condition on $a_{0n}$ is that $x^n$ divides $a_{0n}$.
The elements of $F_{n\alpha}$ coincide with the values of $\nu$ on $K[x,y]$
translated by $n$, so that for each $n\in \NN$ we have $F_{n\alpha}=n+F_0\cong F_0$. Thus (3) of the conclusions of the Proposition follows from Proposition \ref{Theorem5}.
\end{proof}

\par\vskip.5truecm\noindent
Dale Cutkosky,  \hfill Bernard Teissier,\par\noindent
202 Mathematical Sciences Bldg,\hfill  Institut Math\'{e}matique de Jussieu\par\noindent

University of Missouri,\hfill UMR 7586 du CNRS,\par\noindent
Columbia, MO 65211 USA\hfill 175 Rue du Chevaleret, 75013 Paris,France \par\noindent
cutkoskys@missouri.edu\hfill teissier@math.jussieu.fr\par\noindent

\end{document}